\documentclass[11pt,reqno,a4paper]{amsart}

\usepackage[latin1]{inputenc}
\usepackage{tikz}
\usetikzlibrary{shapes,arrows}
\usetikzlibrary{arrows.meta}
\usepackage{forest}
\usepackage{tikz-qtree}

\usetikzlibrary{arrows,shapes,positioning,shadows,trees}

\usepackage{etoolbox}
\usepackage{amsmath}
\usepackage{enumerate}
\usepackage{amssymb}
\usepackage{amscd}
\usepackage{amsthm}
\usepackage{amsfonts}
\usepackage{graphicx}
\usepackage[all,cmtip]{xy}
\usepackage{enumitem}

\patchcmd{\subsection}{-.5em}{.5em}{}{}
\patchcmd{\subsubsection}{-.5em}{.5em}{}{}

\usepackage{enumitem}

\usepackage[T1]{fontenc}

\usepackage[scaled]{helvet} 
\usepackage{courier} 
\usepackage{eulervm}
\normalfont

\makeatother
\usepackage{hyperref}

\numberwithin{equation}{section}

\newcommand{\SL}{\operatorname{SL}}

\newcommand{\cA}{\mathcal{A}}

\newcommand{\cC}{\mathcal{C}}

\newcommand{\cK}{\mathcal{K}}

\newcommand{\cN}{\mathcal{N}}

\newcommand{\cP}{\mathcal{P}}
\newcommand{\cQ}{\mathcal{Q}}
\newcommand{\cR}{\mathcal{R}}

\newcommand{\bF}{\mathbb{F}}

\newcommand{\bR}{\mathbb{R}}

\newcommand{\bZ}{\mathbb{Z}}

\newcommand{\ra}{\rightarrow}

\newcommand{\qen}{\enskip \textrm{and} \enskip}
\newcommand{\qand}{\quad \textrm{and} \quad}

\def\acts{\curvearrowright}
\newcommand\subsetsim{\mathrel{%
\ooalign{\raise0.2ex\hbox{$\subset$}\cr\hidewidth\raise-0.8ex\hbox{\scalebox{0.9}{$\sim$}}\hidewidth\cr}}}
\newcommand{\eps}{\varepsilon}

\DeclareMathOperator{\cum}{cum}

\theoremstyle{theorem}
\newtheorem{theorem}{Theorem}[section]
\newtheorem{corollary}[theorem]{Corollary}
\newtheorem{proposition}[theorem]{Proposition}
\newtheorem{lemma}[theorem]{Lemma}

\theoremstyle{definition}
\newtheorem{definition}[theorem]{Definition}
\newtheorem{remark}[theorem]{Remark}

\tikzstyle{decision} = [diamond, draw, fill=blue!20, 
    text width=4.5em, text badly centered, node distance=3cm, inner sep=0pt]
\tikzstyle{block} = [rectangle, draw, fill=blue!20, 
    text width=5em, text centered, rounded corners, minimum height=2em]
\tikzstyle{line} = [draw, -latex']
\tikzstyle{cloud} = [draw, ellipse,fill=red!20, node distance=3cm,
    minimum height=2em]

\begin{document}

\title[Central Limit Theorems for group actions]{Central limit theorems for group actions which are exponentially mixing of all orders}

\author{Michael Bj\"orklund}
\address{Department of Mathematics, Chalmers, Gothenburg, Sweden}
\email{micbjo@chalmers.se}

\author{Alexander Gorodnik}
\address{University of Bristol, Bristol, UK}
\email{a.gorodnik@bristol.ac.uk}



\date{}


\maketitle

\begin{abstract}
In this paper we establish a general dynamical Central Limit Theorem (CLT) for group actions which are exponentially mixing of all orders.
In particular, the main result applies to Cartan flows on finite-volume quotients of simple Lie groups. Our proof uses a novel relativization
of the classical method of cumulants, which should be of independent interest. As a sample application of our techniques, we show that
the CLT holds along lacunary samples of the horocycle flow on finite-area hyperbolic surfaces applied to any smooth compactly supported function.
\end{abstract}

\section{Introduction}

\subsection{Central limit theorems in dynamics}

One of the fundamental problems in the theory of dynamical systems is to understand whether 
a sequence of observables of a chaotic dynamical system computed along generic orbits behaves 
similarly to a sequence of independent identically distributed random variables. 
More precisely, given a measure-preserving transformation $T$ of a probability space $(X,\mu)$
and a measurable function $f$ on $X$,
one is interested in analysing statistical properties of the sequence
\begin{equation}\label{eq:sequence}
f(Tx), f(T^2x),\ldots,f(T^kx),\ldots 
\end{equation}
where $x\in X$ is distributed according to the measure $\mu$.
One says that the ``Central Limit Theorem'' (CLT) holds if 
there 
exists $\sigma_f \geq 0$ such that
\begin{equation}
\label{CLT000}
\frac{1}{\sqrt{N}} \sum_{k=1}^N (f \circ T^k-\mu(f)) \implies N(0,\sigma_f^2)\quad\hbox{ as $N\to \infty$},
\end{equation}
where $N(0,\sigma_f^2)$ denotes the Gaussian distribution with variance $\sigma_f^2$, and $\implies$ denotes convergence in the sense of distributions. Explicitly this means that for every interval $(a,b)$,
$$
\mu\left(\left\{x\in X:\, a<\frac{1}{\sqrt{N}} \sum_{k=1}^N (f(T^kx)-\mu(f))<b \right\}\right)
\longrightarrow \frac{1}{\sqrt{2\pi\sigma_f}} \int_a^b e^{-u^2/(2\sigma_f)}\, du
$$
as $N\to \infty$. Due to deterministic nature of the sequence \eqref{eq:sequence},
it is known that CLT cannot hold for general functions in $L^1(X,\mu)$, 
and one is interested in finding a ``large'' subspace $\cA$ of functions satisfying 
\eqref{CLT000}.
In applications, the space $X$ is often assumed to be a finite-volume Riemannian manifold, $T : X \ra X$ a smooth map which preserves the volume measure on $X$, and $\cA$ is some linear subspace of continuous functions with prescribed
regularity (H\"older, $\cC^r$ for some $r$, etc.). 
Starting with the pioneering work of Sinai \cite{S} who established 
the Central Limit Theorem for geodesic flows on compact manifolds
with constant negative curvature, this problem 
has been extensively studied for transformations satisfying some hyperbolicity assumptions
\cite{leo,Rat,bow,gh,Den,LeJ,Ch,liv,young,lB0,lB,ClB,MT,Dol,gs}.
We refer to \cite{D,Den,gou,lB2,v} for surveys of this area of research.\\

More generally, we consider a measure-preserving action of a group $H$ on a probability space
$(X,\mu)$. Given a measurable function $f$ on $X$, we obtain a collection of observables
\begin{equation}
\label{eq:sec}
\{f(h^{-1}x):\, h\in H\}\quad \hbox{ with $x\in X$.}
\end{equation}
The aim of this paper is to establish a general Central Limit Theorem for 
the averages
$$
\frac{1}{|F_T|^{1/2}}\sum_{h\in F_T} (f\circ h^{-1}-\mu(f)), \quad F_T\subset H,
$$
for actions of higher-dimensional
(possibly non-commutative) groups. 
Previously, CLT was established in \cite{CC1,CC2,CC3} 
for actions of the group
$H = (\bZ^d,+)$ by automorphisms of compact abelian groups.
Our result holds
for general actions that are ``sufficiently chaotic'', which is manifested by
quantitative estimates on the higher-order correlations
$$
\int_X f_1(h_1^{-1}x)\cdots f_r(h_r^{-1}x)\, d\mu(x)
$$
for $h_1,\ldots, h_r\in H$ and $f_1,\ldots,f_r:X\to \mathbb{R}$.

The multi-parameter averages of this type naturally arise in
number-theoretic problems. 
In \cite{BG} we apply the
developed techniques to analyze the discrepancies of distributions of values 
of products of linear forms.

\subsection{The main result}

Let $H$ be a locally compact group equipped with a left-invariant Haar measure $m_H$ and a proper left-invariant metric $d$. We say that $(H,d)$ has \emph{sub-exponential growth} if
\begin{equation}
\label{defsubexp}
\varlimsup_{r \ra \infty} \frac{\log m_H(B_d(e,r))}{r} = 0,
\end{equation}
where $B_d(e,r)$ denotes the $d$-ball around the identity element $e$ in $G$ of radius $r$. 
It is not hard to show that every locally compact second countable abelian group has sub-exponential growth; a bit more work is required to show
that every locally compact second countable nilpotent group has sub-exponential growth. Furthermore, it is well-known that groups 
of sub-exponential growth are amenable, and thus possess right F\o lner sequences. We recall that a sequence $(F_T)$ of
compact subsets with non-empty interiors in $H$ is \emph{right F\o lner} if 
\[
\lim_{T\to\infty} \frac{m_H(F_T\, \triangle\, F_T h)}{m_H(F_T)} = 0, \quad \textrm{for all $h \in H$},
\] 
where $\triangle$ denotes the symmetric difference of sets. It is easy to show that if $H = (\bR^n,+)$, then any sequence of Euclidean balls of increasing radii forms a F\o lner sequence in $H$. More generally, balls in groups of sub-exponential growth also form F\o lner sequences.
\\

Suppose that the group $H$ acts jointly measurably by measure-preserving maps on a probability measure space $(X,\mu)$. We shall assume that 
there exists an $H$-invariant sub-algebra $\cA$ of $L^\infty(X,\mu)$ and a family $\cN=(N_s)$ of semi-norms on $\cA$, which satisfy some technical 
conditions spelled out in Section \ref{standingass} (see \eqref{mono}--\eqref{Leibniz} below), such that the $H$-action
is \emph{exponentially mixing of all orders} with respect to $d$ and $(\cA,\cN)$
(in the sense of Definition \ref{DEFmixk} below). 
Roughly speaking, this property requires that
$$
\int_X f_1(h_1^{-1}x)\cdots f_r(h_r^{-1}x)\, d\mu(x)=\left(\int_X f_1\,d\mu\right)\cdots \left(\int_X f_r\,d\mu\right)+o(1)
$$
with $f_1,\ldots,f_r\in\cA$ and $h_1,\ldots,h_r\in H$ and an explicit error term depending on the quantities $\min\{d(h_i,h_j):\,i\ne j\}$ and $\mathcal{N}_s(f_i)$.
We refer to Section \ref{standingass} for the required definitions
and notation. \\

Our main result now reads as follows. 

\begin{theorem}[Sub-exponential growth and CLT]
\label{main0}
Let $(H,d)$ have sub-exponential growth and $(F_T)$ be a right F\o lner sequence in $H$ such that $m_H(F_T)\to\infty$. 
Suppose that the action of $H$ on $(X,\mu)$ is 
exponentially mixing of all orders on an $H$-invariant sub-algebra $\cA$ of $L^\infty(X,\mu)$.
Then, for any $f \in \cA$ with $\mu(f) = 0$, 
\begin{equation}
\label{clt0}
\frac{1}{m_H(F_T)^{1/2}} \int_{F_T} (f\circ h^{-1}) \, dm_H(h) \implies N(0,\sigma_f^2) \quad\hbox{ as $T\to\infty$},
\end{equation}
where 
\begin{equation}
\label{sigmaf0}
\sigma_f^2 = \int_H \langle f, f\circ h^{-1} \rangle \, dm_H(h).
\end{equation}
\end{theorem}

A more general version of this theorem, which does not assume  sub-exponential growth,
will be stated in Theorem \ref{main}.

\begin{remark}
Exponential $2$-mixing, combined with the sub-exponential growth of $H$, will ensure that $\sigma_f$ given by \eqref{sigmaf0} is finite for all $f \in \cA$ (see Section \ref{sec:var}). However, 
we stress that $\sigma_f  = 0$ is definitely possible for non-zero $f$; indeed, it is not hard to show that this happens if $f = g - g\circ h_o$ for some $h_o \in H$ and $g \in \cA$. 
\end{remark}

\subsubsection{A sample application: CLT for Cartan actions on homogeneous spaces}\label{sec:hom}

Let $L$ be a connected Lie group, $\Gamma$ a lattice in $L$ and $\mu_Y$ the unique $L$-invariant probability measure
on $Y = \Gamma \backslash L$. Let $G$ be a semisimple Lie subgroup of $L$ with finite center, and assume that the $G$-action 
on $(Y,\mu_Y)$ has strong spectral gap. Let $A < G$ denote a Cartan subgroup of $G$, and fix a closed subgroup 
$H$ of $A$. Let $(B_T)$ be a sequence of strictly increasing balls in $H$ with respect to some left-invariant Riemannian metric on 
$G$ restricted to $H$. One readily checks that $(B_T)$ forms a right F\o lner sequence in $H$.

In recent joint work \cite{BEG} with M. Einsiedler, the authors showed that if one takes $\cA$ to be the algebra of smooth functions on $Y$
with compact supports, and $\cN$ a family of certain Sobolev norms on $\cA$, then the assumptions of Theorem \ref{clt0} are 
satisfied for the $H$-action on $(Y,\mu_Y)$ with respect to $(\cA,\cN)$, which leads to the following corollary of Theorem \ref{main0}.

\label{sample1}
\begin{corollary}
\label{cor0}
For every real-valued, compactly supported smooth function $f$ on $Y$ with $\mu_Y(f) = 0$, 
\[
\frac{1}{m_H(B_T)^{1/2}} \int_{B_T} (f\circ h^{-1}) \, dm_H(h) \implies N(0,\sigma_f^2),
\]
where $\sigma_f$ is given by \eqref{sigmaf0}.
\end{corollary}

For example, Corollary \ref{cor0} applies when
$$
X=\SL_m(\bZ)\backslash \SL_m(\bR)
\quad\hbox{or}\quad X=(\bZ^m \rtimes \SL_m(\bZ))\backslash (\bR^m \rtimes \SL_m(\bR)),
$$
and $H$ is a closed subgroup of the group of diagonal matrices in $\SL_m(\bR)$.

\begin{remark}
The first author and G. Zhang \cite{BZ} have recently provided, under some technical assumptions on $G$ and $H$, lower (positive) bounds on the variance $\sigma_f$, whenever $f$ is non-zero and invariant under a maximal compact subgroup of $G$.
\end{remark}

\subsection{Connections to earlier works}
A number of different approaches have been developed for proving 
dynamical Central Limit Theorems for one-parameter actions.
An influential approaches based on a martingale approximation 
originated in the work of Gordin \cite{Go}.
We refer to the survey \cite{lB2} by Le Borgne for an overview of this technique, 
as well as for an extensive list of references. 
The martingale approximation method becomes harder to implement already for actions of $H = (\bZ^d,+)$, $d \geq 2$, let alone for actions by non-commutative groups; see for instance \cite{Vol} for some recent developments in this direction when $H = (\bZ^d,+)$.
Other approaches to the Central Limit Theorem involve Markov
approximations \cite{Rat,bow,Ch,Dol} and 
spectral analysis of transfer operators \cite{gh,gou},
and it is also not clear how to implement them for multi-parameter actions.

In this paper, we use an alternative approach to the Central Limit Theorem based on the classical \emph{method of cumulants}, 
due to Fr\'echet and Shohat \cite{FS} (see Section \ref{sec:outlineproofmain} for an outline of this method). Roughly speaking, 
this method is equivalent to the more well-known \emph{method of moments}, but is better tailored for approximations to Gaussian laws.
An important novelty in our work is the systematic use of \emph{conditional cumulants} (see Section \ref{sec:cond}), which greatly
simplifies the estimation of cumulants in the presence of exponential mixing of all orders. 

The method of cumulants has been recently used by Cohen and Conze in \cite{CC1,CC2,CC3} to establish Central Limit 
Theorems for multiple mixing actions by $H = (\bZ^d,+)$ by automorphisms of a compact abelian group $X$ --- but only for functions
which are finite sums of characters on $X$. It seems that it is still not known whether in this setting the CLT also holds for all smooth functions. 

\subsection{A general Cental Limit Theorem}

Our method allows to prove the Central Limit Theorem for actions of general groups
and for more general averaging schemes provided that some technical conditions are verified.
Let us now assume that $H$ is  a locally compact second countable group equipped
with a left-invariant metric $d$.

The following is the main technical result of this paper.

\begin{theorem}[General CLT]
\label{main}
Let $(\nu_T)$ be a sequence of positive and finite Borel measures on $H$ such 
that $\|\nu_T\| \ra \infty$, and for any integer $r \geq 3$
and real number $c > 0$, 
\begin{equation}
\label{cond_nuT}
\lim_{T\to\infty} \, \int_H \nu_T(B_d(h,c \log \|\nu_T\|))^{r-1} \, d\nu_T(h) = 0.
\end{equation}
Suppose that the action of $H$ on $(X,\mu)$ is 
exponentially mixing of all orders on an $H$-invariant sub-algebra $\cA$ of $L^\infty(X,\mu)$.
Let $f \in \cA$ with $\mu(f) = 0$ and suppose that the limit
\begin{equation}
\label{cond_sigmaf}
\sigma_f := \lim_{T\to\infty} \, \|\nu_T * f \|_L^2 <\infty
\end{equation}
exists.
Then 
$$
\nu_T * f \implies N(0,\sigma^2_f)\quad\hbox{ as $T\to\infty$.}
$$
\end{theorem}

In Section \ref{sec:thmmain0} below we show that these conditions are satisfied for the averages 
\[
\nu_T = \frac{1}{m_H(F_T)^{1/2}} \int_{F_T} \, \delta_h \, dm_H(h),
\]
when $(H,d)$ has sub-exponential growth, and $(F_T)$ is a right F\o lner sequence in $H$, and 
we identify the limit in \eqref{cond_sigmaf}.
Hence, Theorem \ref{main0} will be deduced from Theorem \ref{main}.

\subsubsection{A sample application: CLT for unipotent flows sampled at lacunary times}

Let us now retain the notation from Section \ref{sec:hom}, 
and let $u(t)$ be a non-trivial one-parameter unipotent subgroup $G$.
Although actions of Cartan subgroups of $G$ 
on $Y = \Gamma \backslash L$ are exponentially mixing of all orders, actions of \emph{unipotent}
subgroups are not (although they are \emph{polynomially mixing of all orders}). Sinai raised the 
question whether the Central Limit Theorem still holds for unipotent flows on $Y$; this was later 
answered in the negative by Flaminio and Forni in \cite[Cor.~1.6]{FF}. 
However, our next corollary shows that if one is willing to ``speed up'' unipotent flows by sampling them at lacunary times, then the CLT does hold.

\begin{corollary}
\label{cor}
For any lacunary sequence $(\lambda_k) \subset \bR_{+}$ and a real-valued, compactly 
supported smooth function $f$ on $Y$ with $\mu_Y(f) = 0$,
\begin{equation}
\label{clt1}
\frac{1}{\sqrt{N}} \sum_{k=1}^{N} f\circ u({\lambda_k}) \implies N(0,\|f\|_{L^2}^2)\quad\hbox{ as $N\to\infty$.}
\end{equation}
\end{corollary}

\begin{remark}
Contrary to Theorem \ref{main0}, we note that in this setting the variance $\sigma_f$ is always positive whenever $f$ has mean zero and does not 
vanish identically on $Y$.
\end{remark}
\subsection{Structure of the paper}

In Sections \ref{sec:thmmain0} and \ref{sec:cormain} below we shall assume Theorem \ref{main} and show how
Theorem \ref{main0} and Corollary \ref{cor} respectively can be deduced from it. 

The rest of the paper is then devoted to the proof of Theorem \ref{main}, which we shall break down into several 
propositions, whose proofs, in turn, will be divided into further lemmas and propositions. The following tree graphically 
represents this break-down:
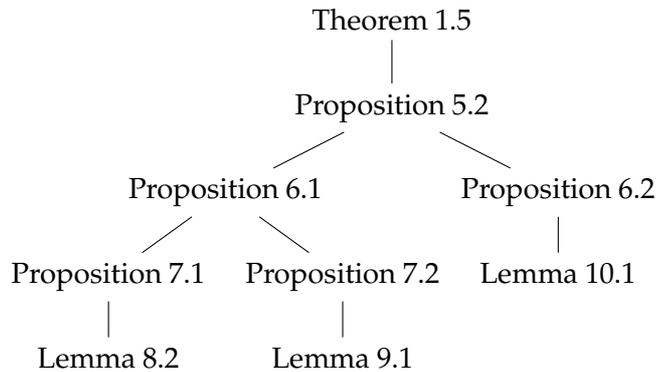
\begin{center}
\begin{forest} 
[Theorem \ref{main}
     [Proposition \ref{mainprop}
            [Proposition \ref{prop_offdiagonal} [Proposition \ref{prop_cond}[Lemma \ref{taumap}]] [Proposition \ref{prop_conditional_estimate} [Lemma \ref{lemma_estimate_I}]]]
     [Proposition \ref{prop_exhaust}
            [Lemma \ref{lemma_coarser}]]
            ]]
\end{forest}
\end{center}

\section{Definitions and standing technical assumptions}
\label{standingass}

Let us here in this section give the definitions of some notions used in the introduction, and collect some of the technical 
assumptions which are necessary for our analysis. \\

Throughout the paper, we shall fix a locally compact second countable group $H$, and a left-invariant metric $d$ on $H$. We assume that $H$ acts jointly
measurably by measure preserving maps on a probability measure space $(X,\mu)$. In particular, $H$ also acts (weakly continuously) by isometries on $L^\infty(X,\mu)$ via 
\[
h \cdot f = f \circ h^{-1}, \quad \textrm{for $f \in L^\infty(X,\mu)$ and $h \in H$}.
\]
We fix an $H$-invariant sub-algebra $\cA$ of $L^\infty(X,\mu)$, and a family $\cN = (N_s)$ of semi-norms on $\cA$, indexed by 
positive integers $s$. 

\begin{definition}[Exponential mixing of all orders]
\label{DEFmixk}
Let $r \geq 2$ be an integer. We say that the $H$-action on $(X,\mu)$ is \emph{exponentially mixing of order $r$}, with respect to $d$ and $(\cA,\cN)$, if there exist $\delta_r > 0$ and an integer 
$s_r > 0$ such that for all $s > s_r$ and $f_1,\ldots, f_r \in \cA$,
\begin{equation}
\label{defmixk}
\Big| 
\mu
\Big(
\prod_{i=1}^r h_i \cdot f_i
\Big) - \prod_{i=1}^r \mu(f_i) 
\Big|
\ll_{r,s}
e^{-\delta_r d_r(\underline{h})} \, \prod_{i=1}^r N_s(f_i),
\end{equation}
for all $\underline{h} = (h_1,\ldots,h_r) \in H^r$, where
\begin{equation}
\label{defdr}
d_r(\underline{h}) = \min_{i \neq j} d(h_i, h_j).
\end{equation}
We refer to $\delta_r$ as the \emph{rate of $r$-mixing}.
The notation $A \ll_{r,s} B$ here means that there exists a constant $c$, which is allowed to depend on $r$ and $s$ such that 
$A \leq c\, B$. 
\end{definition}

We shall from now on always assume that the $H$-action on $(X,\mu)$ is exponentially mixing of all orders with respect to $d$ and $(\cA,\cN)$; in particular, from now on, the meanings of the numbers $s_r$ and $\delta_r$ have been fixed. Furthermore, we shall require 
the following four technical conditions on the family $\cN=(N_s)$
(with all implicit constants depending only on $s$):
\begin{itemize}
\item \textbf{Monotonicity:} For all $s \geq 1$ and $f \in \cA$,
\begin{equation} 
N_s(f) \ll_s N_{s+1}(f). \label{mono}
\end{equation}
\item \textbf{Sobolev embedding:} For all $s \geq 1$ and $f \in \cA$,
\begin{equation}
\|f\|_{L^\infty} \ll_s  N_s(f). \label{unibnd}
\end{equation}
\item \textbf{$H$-boundedness:} For all $s \geq 1$, there exists $\sigma_s > 0$ such that for all  $f \in \cA$ and $h \in H$,
\begin{equation}
N_s(h \cdot f) \ll_s e^{\sigma_s \, d(h,e)} \, N_s(f). \label{grpbnd}
\end{equation}
\item \textbf{Almost multiplicative:} For all $s \geq 1$ and $f_1, f_2 \in \cA$,
\begin{equation}
N_s(f_1 f_2) \ll_s  N_{s+1}(f_1) \, N_{s+1}(f_2). \label{Leibniz}
\end{equation}
\end{itemize}
 We may further, and shall, throughout the paper, assume that the sequence $(\sigma_s)$ increases with $s$, and the sequence $(\delta_r)$ from Definition \ref{DEFmixk} decreases with $r$. Also, without 
any loss of generality, we may assume that $\delta_r < r \sigma_s$ for all $r$ and $s$.

\section{Proof of Theorem \ref{main0} assuming Theorem \ref{main}}
\label{sec:thmmain0}
Recall our standing assumptions on $H, d, (X,\mu)$ and $(\cA,\cN)$ from Section \ref{standingass}. Let us further 
assume that $(H,d)$ has sub-exponential growth, and $(F_T)$ is a right F\o lner sequence in $H$. 
We shall apply Theorem \ref{main} to the sequence $(\nu_T)$ of positive and finite measures on $H$ 
defined by 
\begin{equation}
\label{Case1}
\nu_T  = \frac{1}{m_H(F_T)^{1/2}} \int_{F_T} \delta_h \, dm_H(h).
\end{equation}
One readily checks that $\|\nu_T\| = m_H(F_T)^{1/2}$ for all $T$. In particular,
$\|\nu_T\|\to \infty$ as $T\to\infty$.
In order to prove Theorem \ref{main0}, it suffices to verify conditions \eqref{cond_nuT} and \eqref{cond_sigmaf}.

\subsection{Checking condition \eqref{cond_nuT}}

Writing out condition \eqref{cond_nuT} explicitly in our setting, we see that we must prove that for every integer $r \geq 3$ 
and real number $c > 0$,  
\[
\frac{1}{m_H(F_T)^{r/2}} \int_{F_T} m_H(F_T \cap B_d(h,\frac{c}{2} \log m_H(F_T)))^{r-1} \, dm_H(h) \ra 0,
\]
as $T\to\infty$. The integral is bounded from above by
\[
m_H(F_T) \, m_H(B_d(e,\frac{c}{2} \log m_H(F_T)))^{r-1}.
\]
Hence it suffices to show that for every $r \geq 3$ and real number $c > 0$,
\[
\frac{m_H(B_d(e,\frac{c}{2} \log m_H(F_T)))^{r-1}}{m_H(F_T)^{r/2-1}} \ra 0,
\]
as $T\to\infty$. Since $m_H(F_T) \ra \infty$ and $r/2 > 1$, this readily follows from the sub-exponential growth of $(H,d)$, see \eqref{defsubexp}.

\subsection{Calculating the variance}\label{sec:var}

Upon expanding \eqref{cond_sigmaf} for our choice of $(\nu_T)$ and a fixed $f \in \cA$ with $\mu(f) =0$, we see that 
one has to show that
\begin{equation}
\label{stepsigmaf}
\frac{1}{m_H(F_T)} \int_H \int_H \chi_{F_T}(h_1) \, \chi_{F_T}(h_2) \phi_f(h_1^{-1}h_2) \, dm_H(h_1) \, dm_H(h_2)
 \ra \int_H \phi_f(h) \, dm_H(h),
\end{equation}
where $\phi_f(h) = \langle f, h \cdot f \rangle$. Using left-$H$-invariance of $m_H$, the left-hand side 
can be re-written as
\begin{equation}
\label{lefthand}
\int_H \frac{m_H(F_T \cap F_T h^{-1})}{m_H(F_T)} \, \phi_f(h) \, dm_H(h).
\end{equation}
Since $(F_T)$ is a right F\o lner sequence in $H$, 
\[
\lim_{T\to\infty} \frac{m_H(F_T \cap F_T h^{-1})}{m_H(F_T)} = 1, \enskip \textrm{for all $h \in H$},
\]
whence \eqref{stepsigmaf} follows from the Dominated Convergence Theorem if we can show that $\phi_f$ belongs to 
$L^1(H)$. Since the $H$-action on $(X,\mu)$ is exponentially mixing of order two with respect to $d$ and $(\cA,\cN)$, we have 
$$
|\phi_f(h)| \ll_{s} e^{-\delta_2 d(h,e)} N_s(f)^2\quad\hbox{ for all 
$h \in H$ and $s > s_2$,}
$$
where the implicit constant is independent of $h$. Hence, the following lemma finishes the proof.

\begin{lemma}
\label{integrable}
If $\phi$ is a complex-valued measurable function on $H$ such that for some 
$C, \alpha > 0$,
\begin{equation}
\label{eq:exp}
|\phi(h)| \leq C\,e^{-\alpha d(h,e)},\quad h\in H,
\end{equation}
then $\phi$ belongs to $L^1(H)$.
\end{lemma}

\begin{proof}
Using \eqref{eq:exp}, we obtain
\[
\int_{H\backslash\{0\}} |\phi(h)| \, dm_H(h) = \sum_{n \geq 0} \int_{B_d(e,n+1) \setminus B_d(e,n)}  |\phi(h)| \, dm_H(h)
\leq C \sum_{n \geq 0} \beta_n e^{-\alpha n}, 
\]
where $\beta_n = m_H(B_d(e,n+1))$. Since $(H,d)$ has sub-exponential growth, $\beta_n^{1/n} \ra 1$, 
which readily implies that the series converges.
\end{proof}

\section{Proof of Corollary \ref{cor} assuming Theorem \ref{main}}
\label{sec:cormain}
Recall our assumptions on $L, \Gamma, G, Y$ and $\mu_Y$ from Corollary \ref{cor}, and let $u(t)$ be a non-trivial unipotent one-parameter subgroup of $G$, and $(\lambda_k)$ a sequence in $\bR_{+}$ such that for some $\theta > 1$,
\begin{equation}
\label{lac}
\lambda_{k+1} \geq \theta \lambda_k, \quad \textrm{for all $k$}. 
\end{equation}
We shall apply Theorem \ref{main} to the sequence $(\nu_N)$ of positive and finite measures on $H$ 
defined by 
\begin{equation}
\label{Case2}
\nu_N = \frac{1}{\sqrt{N}} \sum_{k=1}^{N} \delta_{u({\lambda_k})^{-1}}.
\end{equation}
One readily checks that $\|\nu_N\| = \sqrt{N}$. 

We shall crucially use the easily 
checkable fact
the distance along unipotent subgroups grows at least logarithmically
(see, for instance, \cite[Lem.~2.1]{BEG}):
for every 
fixed choice of such a unipotent subgroup $u(t)$, there exist $c_1,c_2> 0$ such that
\begin{equation}
\label{logdist}
d(u(t),e) \geq c_1\,\log|t|-c_2\quad \hbox{ for all $t\ne 0$.}
\end{equation}
In particular, it follows from the exponential mixing property that
there exists $p>0$ such that for every $s>s_2$ and $f_1,f_2\in \cA$ satisfying $\mu_Y(f_1)=\mu_Y(f_2)=0$,
\begin{equation}
\label{eq:mix_u}
\left<f_1,u(t)\cdot f_2\right> \ll_s |t|^{-p} N_s(f_1)N_s(f_2)\quad \hbox{ for all $t\ne 0$.}
\end{equation}

\subsection{Checking condition \eqref{cond_nuT}}

If we write out condition \eqref{cond_nuT} explicitly in our setting, we see that we must show that for all $r \geq 3$ and $c > 0$,
\[
\frac{1}{N^{r/2}} \sum_{n=1}^N \big| \{ m=1,\ldots,N \, : \, d(u(\lambda_m),u(\lambda_n)) \leq c \log \sqrt{N} \}\big|^{r-1} \ra 0,
\]
as $N \ra \infty$. We deduce from \eqref{logdist} that for $m\ne n$,
$$
d(u(\lambda_m),u(\lambda_n))=d(u(\lambda_m-\lambda_n),e)\ge c_1\,\log|\lambda_m-\lambda_n|-c_2.
$$
Hence, it suffices to show that for every $c,\theta>0$,
\begin{equation}\label{s}
\frac{1}{N^{r/2}} \sum_{n=1}^N \big| \{ m=1,\ldots,N \, : \, |\lambda_m-\lambda_n| \leq c\,N^\theta \}\big|^{r-1} \ra 0.
\end{equation}
Using \eqref{lac}, it is not hard to show that 
$$
\big| \{ m\ge 1 \, : \, |\lambda_m-\lambda_n| \leq c\,N^\theta \}\big|\ll \log N,
$$
where the implied constant is independent of $n$. Then
$$
\sum_{n=1}^N \big| \{ m=1,\ldots,N \, : \, |\lambda_m-\lambda_n| \leq c\,N^\theta \}\big|^{r-1}
\ll N(\log N)^{r-1},
$$
and \eqref{s} is immediate since $r \geq 3$.

\subsection{Calculating the variance}

Take $f \in \cA$ with $\mu_Y(f)=0$. If we expand $\| \nu_N * f\|_{L^2}^2$, we get
\[
\|f\|_{L^2}^2 + \frac{2}{N} \sum_{1 \leq m < n \leq N} \langle f, u(\lambda_n-\lambda_m) \cdot f \rangle.
\]
We wish to prove that the second term tends to zero as $N \ra \infty$. By \eqref{eq:mix_u}, 
we have for all $s > s_2$,
\[
\langle f, u(\lambda_n-\lambda_m) \cdot f \rangle \ll_{s} \frac{1}{|\lambda_n - \lambda_m|^p} N_s(f)^2.
\]
It thus suffices to show that 
\[
\sum_{1 \leq m < n < \infty} \frac{1}{|\lambda_n - \lambda_m|^p} = \sum_{m=1}^{\infty} \sum_{k=1}^\infty \frac{1}{|\lambda_{m+k}-\lambda_m|^p}< \infty.
\]
Since $|\lambda_{m+k}-\lambda_m| \geq \lambda_m (\theta^k - 1)$ by \eqref{lac}, this follows from the finiteness of the series
\[
\sum_{m=1}^\infty \frac{1}{\lambda_m^p} \qand \sum_{k=1}^\infty \frac{1}{(\theta^{k}-1)^p}.
\]

\section{An outline of the proof of Theorem \ref{main}}
\label{sec:outlineproofmain}

Our proof of Theorem \ref{main} makes use of the classical \emph{cumulant method}, in essence due 
to Fr\'echet and Shohat in \cite{FS}. We shall briefly summarize its main steps below. \\

Let $(X,\mu)$ be a probability measure space and $r \geq 2$ an integer. Denote by $[r]$
the set $\{1,\ldots,r\}$. A \emph{cyclically ordered partition} $\cP$ of the set $[r]$ is a partition $\{I_1,\ldots,I_k\}$ of $[r]$
into non-empty subsets $I_1,\ldots,I_k$, where the cyclic order of $I_1,\ldots,I_k$ is also taken into account; for instance, 
if $[r] = I_1 \sqcup I_2, \sqcup I_3$, then $\{I_1,I_2,I_3\}$ and $\{I_2,I_3,I_1\}$ are viewed as the same \emph{cyclically ordered} partition, while $\{I_1,I_2,I_3\}$ and $\{I_2,I_1,I_3\}$ are viewed as different partitions, since the associated orders $(123)$ and
$(213)$ are not cyclic permutations of each other. We denote by $\mathfrak{P}_{[r]}$ the
set of all cyclically ordered partitions of $[r]$.

Given an $r$-tuple $(f_1,\ldots,f_r)$ in $L^\infty(X,\mu)$ and a subset $I \subset [r]$, we define
\[
f_I = \prod_{i \in I} f_i, \quad \textrm{for $\emptyset \neq I \subset [r]$},
\]
and the \emph{joint cumulant $\cum_{[r]}(f_1,\ldots,f_r)$} of order $r$ by
\begin{equation}
\label{defcumbracketr}
\cum_{[r]}(f_1,\ldots,f_r) = \sum_{\cP \in \mathfrak{P}_{[r]}} (-1)^{|\cP|-1} \prod_{I \in \cP} \mu\big(f_I),
\end{equation}
where $|\cP|$ denotes the number of partition elements in $\cP$. If $f \in L^\infty(X,\mu)$, we define $\cum_r(f)$, 
the \emph{cumulant of $f$ of order $r$}, to be
\begin{equation}
\label{defcumr}
\cum_r(f) = \cum_{[r]}(f,\ldots,f).
\end{equation}
The utility of cumulants for problems pertaining to Central Limit Theorems is well-known; we shall use the following
classical criterion, which can be deduced from the results in \cite{FS}.

\begin{proposition}[Cumulants and CLT]
\label{CLT}
Let $(Z_T)$ be a sequence of  real-valued, bounded and measurable functions on $(X,\mu)$
satisfying $\mu(Z_T)=0$. 
If 
\begin{equation}
\label{rgeq3}
\lim_{T\to\infty} \cum_r(Z_T) = 0, \quad \textrm{for all $r \geq 3$},
\end{equation}
and the limit
\begin{equation}
\label{req2}
\sigma^2 := \lim_{T\to\infty} \|Z_T\|_{L^2}^2 < \infty
\end{equation}
exists, then 
$$
Z_T \implies N(0,\sigma^2)\quad\hbox{ as $T\to\infty$.}
$$
\end{proposition}

\subsection{Main proposition}

Recall our standing assumptions concerning $H, d, (X,\mu), \cA$ and the norms $(N_s)$ from Subsection \ref{standingass}. 
In particular, the meaning of the numerical sequences $(s_r)$ and $(\delta_r)$ has been fixed.
Let $(\nu_T)$ be a sequence of positive and finite measures on $H$. Given $f \in \cA$, thanks to Proposition \ref{CLT}, 
the proof of the CLT for 
the sequence $Z_T := \nu_T * f$ is essentially reduced to the asymptotic vanishing of $\cum_r(\nu_T * f)$ for $r \geq 3$. Using
our assumptions on $(\nu_T)$ in Theorem \ref{main}, this will be deduced from the following proposition.

\begin{proposition}[Estimating cumulants]
\label{mainprop}
For all $r \geq 3$ and $s > s_r + r$, there exists $c_{r,s} > 0$ such that for all $\gamma > 0$, $T > 0$ and $f \in \cA$,
\begin{equation}
\label{cumbnd0}
|\cum_r(\nu_T * f)| \ll_{r,s}
 \Big(\int_H \nu_T(B_d(h,c_{r,s} \gamma))^{r-1} \, d\nu_T(h) + e^{-\delta_r \gamma} \|\nu_T\|^r \Big) 
 \, N_s(f)^r
\end{equation}
where the implicit constant depends only on $r$ and $s$.
\end{proposition}

\subsection{Proof of Theorem \ref{main} assuming Theorem \ref{mainprop}}

Fix $f \in \cA$. By Proposition \ref{CLT}, applied to $Z_T = \nu_T * f$, we need to show that the limit
\[
\sigma_f^2 = \lim_{T\to\infty} \|\nu_T *f \|_{L^2}^2 
\]
exists, and
\[
\lim_{T\to\infty} \cum_r(\nu_T * f) = 0, \enskip \textrm{for all $r \geq 3$}.
\]
The existence of the first limit is assumed in Theorem \ref{main}, so we only need to consider the 
second kind of limits. Fix $r \geq 3$, and choose $\gamma = \eps \log \|\nu_T\|$ in Proposition \ref{mainprop}
for some $\eps > r/\delta_r$. Then, \eqref{cumbnd0} yields for all $s > s_r + r$,
\[
|\cum_r(\nu_T * f)| \ll_{r,s}
\Big(\int_H \nu_T(B_d(h,c_{r,s} \eps \log \|\nu_T\|))^{r-1} \, d\nu_T(h) + \|\nu_T\|^{-(\delta_r \eps - r)} \Big) 
 \, N_s(f)^r.
\]
Applying our assumption \eqref{cond_nuT} with $c = c_{r,s} \eps$,
we conclude that this quantity tends to zero as $T \ra \infty$, since $\|\nu_T\| \ra \infty$. 

\section{An outline of the proof of Proposition \ref{mainprop}}

Throughout this section, we retain our assumptions on $H, d, (X,\mu), \cA$ and the norms $(N_s)$.
We shall further fix an integer $r \geq 3$, and write $[r]$ for the set $\{1,\ldots,r\}$.

\subsection{Rewriting cumulants}
\label{subsec:rewrite}
Let $(\nu_T)$ be a sequence of positive and finite measures on $H$. For any $f \in \cA$, we see that
\[
\cum_r(\nu_T * f) = \int_{H^r} \cum_{[r]}(h_1 \cdot f,\ldots,h_r \cdot f) \, d\nu_T^{\otimes r}(h_1,\ldots,h_r).
\] 
Furthermore, for every fixed $\underline{h} = (h_1,\ldots,h_r) \in H^r$, we can write every joint cumulant 
of the form $\cum_{[r]}(h_1 \cdot f,\ldots, h_r \cdot f)$ as 
\[
\sum_{\cP\in \mathfrak{P}_{[r]}} (-1)^{|\cP|-1} \prod_{I \in \cP} \psi_{f,\underline{h}}(I) ,
\]
where
\[
\psi_{f,\underline{h}}(I) = \mu\big( \prod_{i \in I} h_i \cdot f \big), \quad \textrm{for $I \subset [r]$}.
\]
We shall now adopt the following notational convention. If $\psi$ is a real-valued function defined on the set $2^{[r]}$
of all subsets of $[r]$ with $\psi(\emptyset) = 1$, then we define its \emph{cumulant} as
\begin{equation}
\label{defcumpsi}
\cum_{[r]}(\psi) = \sum_{\cP\in \mathfrak{P}_{[r]}} (-1)^{|\cP|-1} \prod_{I \in \cP} \psi(I),
\end{equation}
so that $\cum_{[r]}(\psi_{f,\underline{h}}) = \cum_{[r]}(h_1 \cdot f,\ldots,h_r \cdot f)$. With this convention, we now have
\begin{eqnarray}
\cum_r(\nu_T * f) 
&=& 
\int_{H^r} \Big( \sum_{\cP\in \mathfrak{P}_{[r]}} (-1)^{|\cP|-1} \prod_{I \in \cP} \psi_{f,\underline{h}}(I) \Big) \, d\nu^{\otimes r}_T(\underline{h}) \nonumber \\
&=&
\int_{H^r} \cum_{[r]}(\psi_{f,\underline{h}}) \, d\nu^{\otimes r}_T(\underline{h}). \label{altexp}
\end{eqnarray}
In what follows, we shall estimate $\cum_{[r]}(\psi_{f,\underline{h}})$ for ``well-separated'' $r$-tuples $\underline{h} = (h_1,\ldots,h_r)$, and we shall also show that ``most'' $r$-tuples are well-separated on suitable scales. In order to make the notions of ``well-separateness''
and ``most'' more precise, we must first introduce some additional notation.

\subsection{Well-separated $r$-tuples}
\label{subsec:wellsep}
If $I, J \subset [r]$ and $\underline{h} = (h_1,\ldots,h_r) \in H^r$,
we set
\begin{equation}
\label{defdupI}
d^{I}(\underline{h}) = \max\big\{ d(h_i,h_j) \, : \, i,j \in I \big\},
\end{equation}
and
\[
d_{I,J}(\underline{h}) = \min\big\{ d(h_i,h_j) \, : \, i \in I, \enskip j \in J \big\}.
\]
If $\cQ$ is a partition of $[r]$, we set
\begin{equation}\label{eq:dQ}
d^{\cQ}(\underline{h}) = \max\big\{ d^{I}(\underline{h}) \, : \, I \in \cQ \big\}
\end{equation}
and
\[
d_{\cQ}(\underline{h}) = \min\big\{ d_{I,J}(\underline{h}) \, : \, I \neq J, \enskip I, J \in \cQ \big\}.
\]
Two extremal cases of this notation will be of special interest. We note that if we write $\cK_r$ for the partition of $[r]$ into points, then
\[
d_{\cK_r}(\underline{h}) = \min\big\{ d(h_i,h_j) \, : \, 1\le i \neq j \le r\big\},
\]
which we have previously also denoted by $d_r(\underline{h})$. At the other extreme, if $\{ [r] \}$ denotes the 
partition into one single block, then
\[
d^{\{[r]\}}(\underline{h}) = \max\big\{ d(h_i,h_j) \, : \, 1\le i, j\le r \big\}.
\]
In order to ease the
somewhat heavy notation, we set
$d^{r} = d^{\{[r]\}}.$
For $\beta>0$, we set
$$
\Delta(\beta)= 
\big\{ 
\underline{h} \in H^r \, : \, 
d^{r}(\underline{h}) \leq \beta \big\}.
$$
We note that for all $h_1 \in H$, the set
$$
\Delta(\beta)_{h_1} := \big\{ (h_2,\ldots,h_r) \, : \, (h_1,\ldots,h_r) \in \Delta(\beta) \big\}  
$$
satisfies
\begin{equation}
\label{almostdiaguse}
\Delta(\beta)_{h_1}
\subset B_d(h_1,\beta)^{r-1}.
\end{equation}
Given a partition $\cQ$ of $[r]$, and $0 \leq \alpha < \beta$, we define
\begin{equation}
\label{defddownQ}
\Delta_{\cQ}(\alpha,\beta) 
= 
\big\{ 
\underline{h} \in H^r \, : \, 
d^{\cQ}(\underline{h}) \leq \alpha, 
\qen
d_{\cQ}(\underline{h}) > \beta \big\}.
\end{equation}
We shall think of the elements in $\Delta_{\cQ}(\alpha,\beta)$ for some partition $\cQ$ with $|Q|\ge 2$ and $0 \le \alpha < \beta$ as being
``well-separated'', while we think of the elements in $\Delta(\beta)$ as being ``clustered''.  
\subsection{Main propositions}

Our first proposition roughly asserts that the joint cumulants $\cum_{[r]}(h_1 \cdot f,\ldots,h_r \cdot f)$ are ``small'' 
for all ``well-separated'' $r$-tuples $\underline{h} = (h_1,\ldots,h_r)$.

\begin{proposition}[Separated tuples]
\label{prop_offdiagonal}
Let $\cQ$ be a partition of $[r]$ with $|\cQ| \geq 2$, and fix $0 \leq \alpha < \beta$ and an integer $s > s_r + r$.
Then, for every $\underline{h} \in \Delta_{\cQ}(\alpha,\beta)$ and $f \in \cA$,
we have
\begin{equation}
\label{offdiagonal}
|\cum_{[r]}(\psi_{f,\underline{h}})| \ll_{r,s} e^{-(\beta \delta_r - r \alpha \sigma_s)} \, N_s(f)^r,
\end{equation}
where the implicit constant depends only on $r$ and $s$.
\end{proposition}

Our second proposition roughly shows that we have a lot of flexibility in setting up the thresholds for the notions of 
``well-separated'' and ``clustered''.

\begin{proposition}[Exhausting $H^r$]
\label{prop_exhaust}
For every sequence $(\beta_j)$ with $\beta_o = 0$ and
\begin{equation}
\label{betaj}
0 < \beta_1 < 3 \beta_1 < \beta_2 < \cdots < \beta_{r-1} < 3 \beta_{r-1} < \beta_r,
\end{equation}
we have
\begin{equation}
\label{exhaust}
H^r = \Delta(\beta_r) \cup \Big( \bigcup_{j=0}^{r-1} \bigcup_{|\cQ| \geq 2} \Delta_{\cQ}(3\beta_j,\beta_{j+1}) \Big).
\end{equation}
\end{proposition}

\subsection{Proof of Proposition \ref{mainprop} assuming Proposition \ref{prop_offdiagonal} and Proposition \ref{prop_exhaust}}

Fix $s > s_r + r$ and $f \in \cA$, and pick a sequence $(\beta_j)$ such that $\beta_o = 0$, and 
\begin{equation}
\label{betaj}
0 < \beta_1 < 3 \beta_1 < \beta_2 < 3 \beta_2 < \beta_3 < \ldots < \beta_{r-1} < 3 \beta_{r-1} < \beta_{r}.
\end{equation}
By \eqref{altexp} and \eqref{exhaust}, we have that for all $T > 0$, 
\begin{eqnarray}
|\cum_r(\nu_T * f)| &\ll_r & \nu^{\otimes r}_T(\Delta(\beta_r)) \, \|f\|_{L^\infty}^r \nonumber \\
&+& 
\max_{j} \, \max_{|\cQ| \geq 2} \int_{\Delta_\cQ(3\beta_j,\beta_{j+1})} |\cum_{[r]}(\psi_{f,\underline{h}})| \, d\nu^{\otimes r}_T(\underline{h}), \label{cumbnd}
\end{eqnarray}
where the second maximum is taken over all partitions of $[r]$ with at least two partition elements. Recall
from our standing assumptions in Subsection \ref{standingass} that $\|f\|_\infty \ll_s N_s(f)$. Hence, using the inclusion  \eqref{almostdiaguse} for the first term, we see that
\begin{eqnarray}
|\cum_r(\nu_T * f)| &\ll_r & \Big( \int_H \nu_T(B(h,\beta_r))^{r-1} \, d\nu_T(h) \Big) \, N_s(f)^r  \nonumber \\
&+& 
\max_{j} \, \max_{|\cQ| \geq 2} \int_{\Delta_\cQ(3\beta_j,\beta_{j+1})} |\cum_{[r]}(\psi_{f,\underline{h}})| \, d\nu^{\otimes r}_T(\underline{h}). \label{cumbnd2}
\end{eqnarray}
We stress that this inequality is valid for any $f \in \cA$ and sequence $(\beta_j)$ satisfying \eqref{betaj}. It remains to choose
a sequence $(\beta_j)$ so that the second term is as small as possible. \\

Let us now fix $\gamma > 0, T > 0$ once and for all, and set $\beta_o = 0$. For every $j \geq 0$, we pick recursively $\beta_{j+1}$ so that
\begin{equation}
\label{betaj2}
\beta_{j+1} \delta_r - 3r \beta_j \sigma_s = \delta_r \gamma.
\end{equation}
Recall from Subsection \ref{standingass} that we assume that $\delta_r < r \sigma_s$, and thus \eqref{betaj2} in particular implies that $3 \beta_j < \beta_{j+1}$, that is to say, $(\beta_j)$ thus constructed satisfies \eqref{betaj}. By induction, \eqref{betaj2} also implies that
\begin{equation}
\label{defcr}
\beta_r \leq \gamma \, \sum_{j=0}^{r-1} \Big( \frac{3r\sigma_s}{\delta_r} \Big)^j =: \gamma c_{r,s}.
\end{equation}
In what follows, we can choose any sequence $(\beta_j)$ as in \eqref{betaj2} with $\beta_r \le \gamma c_{r,s}$. \\

We fix a partition $\cQ$ of $[r]$ with $|\cQ| \geq 2$ and an index $j$. By Proposition \ref{prop_offdiagonal}, we know that for all ``well-separated'' $r$-tuples $\underline{h} \in \Delta_{\cQ}(3\beta_j,\beta_{j+1})$,
\[
|\cum_{[r]}(\psi_{f,\underline{h}})| \ll_{r,s} e^{-(\beta_{j+1} \delta_r - 3r \beta_j \sigma_s)} \, N_s(f)^r
= e^{-\delta_r \, \gamma} \, N_s(f)^r,
\]
where the last equality follows from \eqref{betaj2}. We stress that the right hand side is independent of both $\cQ$ and $j$, and
thus it follows from \eqref{cumbnd2} that
\[
|\cum_r(\nu_T * f)| \ll_{r,s} \Big( \int_H \nu_T(B(h,c_{r,s} \gamma))^{r-1} \, d\nu_T(h)  + e^{-\delta_r \, \gamma} \Big) \, N_s(f)^r,
\]
which finishes the proof.

\section{Proof of Proposition \ref{prop_offdiagonal}}

We retain the notation from the previous section. In particular, an integer $r \geq 3$ has been fixed, and
we write $[r]$ for the set $\{1,\ldots,r\}$, and $\mathfrak{P}_{[r]}$ for the set of cyclically ordered partitions
of $[r]$. \\

Recall from Subsection \ref{subsec:rewrite} that if $\psi : 2^{[r]} \ra \bR$ with $\psi(\emptyset) =1$, 
then its cumulant $\cum_{[r]}(\psi)$ is defined by
\[
\cum_{[r]}(\psi) = \sum_{\cP \in \mathfrak{P}_{[r]}} (-1)^{|\cP|-1} \prod_{I \in \cP} \psi(I).
\]
We can extend $\psi$ to a function $\widetilde{\psi} : \mathfrak{P}_{[r]} \ra \bR$ by
\begin{equation}
\label{defextpsi}
\widetilde{\psi}(\cP) = \prod_{I \in \cP} \psi(I), \quad \textrm{for $\cP \in \mathfrak{P}_{[r]}$},
\end{equation}
so that 
$$
\cum_{[r]}(\psi) = \sum_{\cP\in \mathfrak{P}_{[r]}} (-1)^{|\cP|-1} \widetilde{\psi}(\cP).
$$
Finally, given a partition $\cQ$ of $[r]$, we set 
\begin{equation}
\label{defpsiQ}
\psi^{\cQ}(I) = \prod_{J \in \cQ} \psi(I \cap J) \qand \widetilde{\psi}^{\cQ}(\cP) = \prod_{I \in \cP} \psi^{\cQ}(I).
\end{equation}

\subsection{Estimating cummulants}

Recall that if $f \in \cA$ and $\underline{h} = (h_1,\ldots,h_r) \in H^r$, then $\psi_{f, \underline{h}} : 2^{[r]} \ra \bR$ is defined by
\[
\psi_{f,\underline{h}}(I) = \mu(\prod_{i \in I} h_i \cdot f), \quad \textrm{for $\emptyset \neq I \subset [r]$},
\]
and $\psi_{f,\underline{h}}(\emptyset) = 1$. Our first proposition asserts that in order to estimate $\cum_{[r]}(\psi_{f,\underline{h}})$ 
from above, it suffices to estimate all differences of the form 
$|\widetilde{\psi}_{f,\underline{h}}(\cP) -\widetilde{\psi}^{\cQ}_{f,\underline{h}}(\cP)|$, where $\cQ$ varies over all possible 
partitions of $[r]$ with at least two blocks. These differences will be estimated below using our assumption that the $H$-action on
$(X,\mu)$ is exponentially
mixing of all orders. 

\begin{proposition}
\label{prop_cond}
For any partition $\cQ$ of $[r]$ with $|\cQ| \geq 2$, $\underline{h} \in H^r$ and $f \in \cA$, 
\begin{equation}
\label{trickcum}
|\cum_{[r]}(\psi_{f,\underline{h}})| \ll_{r} \max
\left\{ |\widetilde{\psi}_{f,\underline{h}}(\cP) - \widetilde{\psi}^{\cQ}_{f,\underline{h}}(\cP)| :\, \cP \in \mathfrak{P}_{[r]} \right\}.
\end{equation}
\end{proposition}

Given this result, which will be established in Section \ref{sec:cond} below, Proposition \ref{prop_offdiagonal} follows immediately
from the following proposition, which will be established in Section \ref{sec:condest}.

\begin{proposition}[Estimating the effect of conditioning]
\label{prop_conditional_estimate}
Fix $0 \leq \alpha < \beta$ and an integer $s > s_r + r$. Then for any partition $\cQ$ of $[r]$,
$\underline{h} \in \Delta_{\cQ}(\alpha,\beta)$ and $f \in \cA$, 
\begin{equation}
\label{estimatediff}
 |\widetilde{\psi}_{f,\underline{h}}(\cP) - \widetilde{\psi}^{\cQ}_{f,\underline{h}}(\cP)|
\ll_{r,s}
e^{-(\beta \delta_r - r \alpha \sigma_s)} \, N_s(f)^r,
\end{equation}
where the implicit constant depends only on $r$ and $s$.
\end{proposition}

\section{Proof of Proposition \ref{prop_cond}}
\label{sec:cond}

Throughout this section, let $r \geq 3$ be an integer, and write $[r]$ for the set $\{1,\ldots,r\}$. \\

If $\psi : 2^{[r]} \ra \bR$ is a set function with $\psi(\emptyset) =1$, recall the definition of its cumulant $\cum_{[r]}(\psi)$ from 
\eqref{defcumpsi}, and if $\cQ$ is a partition of $[r]$, recall the definition of the ``conditional'' set function
$\psi^{\cQ}$ from \eqref{defpsiQ}, and the definitions of the ``extended'' versions $\widetilde{\psi}$ and $\widetilde{\psi}^{\cQ}$
from \eqref{defextpsi}. It follows immediately from the definition of $\cum_{[r]}(\psi)$ that
\begin{equation}\label{diff0}
|\cum_{[r]}(\psi) - \cum_{[r]}(\psi^{\cQ})| \ll_r \max\left\{ |\widetilde{\psi}(\cP) - \widetilde{\psi}^\cQ(\cP)| \, : \, \cP \in \mathfrak{P}_{[r]} \right\}.
\end{equation}

We recall that the cumulant of random variables $X_1,\ldots, X_r$ vanish provided that
there exists a non-trivial partition $[r]=I\sqcup J$ such that $(X_i:\, i\in I)$
and $(X_j:\, j\in J)$ are independent (see, for instance, \cite[Lem.~4.1]{speed}).
The following proposition is a combinatorial version of this property.
It can be proved by modifying the argument from \cite{speed}
(see also Theorem 2 in \cite{APU}). We include a proof for completeness.

\begin{proposition}
\label{prop_condiszero}
For any partition $\cQ$ of $[r]$ with $|\cQ| \geq 2$ and $\psi : 2^{[r]} \ra \bR$, we have 
$$
\cum_{[r]}(\psi^{\cQ}) = 0.
$$
\end{proposition}

We note that Proposition \ref{prop_cond} is a direct consequence of Proposition
\ref{prop_condiszero} and estimate \eqref{diff0}.

Let us briefly explain the driving mechanism in the proof of this proposition. Recall that $\mathfrak{P}_{[r]}$ denotes the set
of all cyclically ordered partitions of the set $[r]$. Let $\psi : 2^{[r]} \ra \bR$ be a set function and suppose that there exists a 
bijection $\tau : \mathfrak{P}_{[r]} \ra \mathfrak{P}_{[r]}$ such that 
\begin{equation}
\label{taumech}
|\tau(\cP)| = |\cP| + 1 \mod 2 \qand \widetilde{\psi}(\tau(\cP)) = \widetilde{\psi}(\cP),
\end{equation}
for all $\cP \in \mathfrak{P}_{[r]}$. Then,
\[
\cum_{[r]}(\psi) = 
\sum_{\cP \in \mathfrak{P}_{[r]}} (-1)^{|\tau(\cP)|-1} \widetilde{\psi}(\tau(\cP)) 
= 
- \sum_{\cP \in \mathfrak{P}_{[r]}} (-1)^{|\cP|-1} \widetilde{\psi}(\cP) = - \cum_{[r]}(\psi),
\]
and thus $\cum_{[r]}(\psi) = 0$. \\

The next lemma shows that one can produce, for every partition $\cQ$ of $[r]$ with at least
two partition elements, a bijection $\tau : \mathfrak{P}_{[r]} \ra \mathfrak{P}_{[r]}$ such that \eqref{taumech} holds for 
$\widetilde{\psi}^{\cQ}$, for every choice of set function $\psi : 2^{[r]} \ra \bR$. In particular, by the comment above, this shows 
that $\cum_{[r]}(\widetilde{\psi}^{\cQ}) = 0$, which finishes the proof of Proposition \ref{prop_cond}.

\begin{lemma}
\label{taumap}
For any partition $\cQ$ of $[r]$ with $|\cQ| \geq 2$, there exists a bijection $\tau : \mathfrak{P}_{[r]} \ra \mathfrak{P}_{[r]}$ such that 
\begin{itemize}
\item for every $\cP \in \mathfrak{P}_{[r]}$, we have $|\tau(\cP)| = |\cP| + 1$ mod $2$, and
\item for every $\psi : 2^{[r]} \ra \bR$, we have $\widetilde{\psi}^{\cQ} \circ \tau = \widetilde{\psi}^{\cQ}$.
\end{itemize}
\end{lemma}

\subsection{Proof of Lemma \ref{taumap}}

If $\cQ = (I_1,\ldots,I_n)$ is a cyclically ordered partition with $n \geq 2$, we set
\[
M = I_1 \qand N = \bigsqcup_{k=2}^n I_k,
\]
so that $[r] = M \sqcup N$, and we note if $\psi : 2^{[r]} \ra \bR$, then the function $\phi := \psi^{\cQ}$ satisfies
\begin{equation}
\label{dis}
\phi(I \sqcup J) = \phi(I) \, \phi(J), \quad \textrm{for all $I \subset M$ and $J \subset N$}.
\end{equation}
We shall use the decomposition $[r] = M \sqcup N$ to construct a bijection $\tau : \mathfrak{P}_{[r]} \ra \mathfrak{P}_{[r]}$ such that 
\[
|\tau(\cP)| = |\cP| + 1 \mod 2 \qen \widetilde{\phi} \circ \tau = \widetilde{\phi},
\]
for all cyclically ordered partitions $\cP$ of $[r]$ and all functions $\phi : 2^{[r]} \ra \bR$ which satisfy \eqref{dis}. \\

To this end, 
we choose once and for all an element $y_o \in M$. Given a cyclically ordered partition $\cP = (P_1,\ldots,P_k)$ of $[r]$, 
let $i$ be the unique index such that $y_o \in P_i$, and pick the first index $j$ following $i$ (in the cyclic ordering of 
$\cP$) such that $P_j \cap N \neq \emptyset$. We now set
\[
\tau(\cP) = 
\left\{
\begin{array}{ll}
(P_1,\ldots,P_{j-1},P_j \cap M,P_j \cap N,\ldots) & \textrm{if $P_j \cap M \neq \emptyset$}, \\
(P_1,\ldots,P_{j-2},P_{j-1} \sqcup P_j,P_{j+1},\ldots) & \textrm{if $P_j \cap M = \emptyset$.}
\end{array}
\right.
\]
Let $\phi : 2^{[r]} \ra \bR$ be a function which satisfies \eqref{dis}. If $P_j \cap M \neq \emptyset$, we see
that $|\tau(\cP)| = |\cP| + 1$ and by \eqref{dis},
\[
\widetilde{\phi}(\tau(\cP)) = \phi(P_1) \cdots \phi(P_j \cap M) \phi(P_j \cap N) \cdots \phi(P_k)
= \phi(P_1) \cdots \phi(P_k) = \widetilde{\phi}(\cP).
\]
If $P_j \cap M = \emptyset$, we see that $|\tau(\cP)| = |\cP| - 1$.
In this case, we observe that $P_j\subset N$ and $i<j$ because $y_o\in P_i$, so that $P_{j-1}\cap N=\emptyset$ and
$P_{j-1}\subset M$. Hence, by \eqref{dis},
\[
\widetilde{\phi}(\tau(\cP)) = \phi(P_1) \cdots \phi(P_{j-2}) \phi(P_{j-1} \sqcup P_j) \cdots \phi(P_k)
= \phi(P_1) \cdots \phi(P_k) = \widetilde{\phi}(\cP).
\]
It is clear that the map $\tau : \mathfrak{P}_{[r]} \ra \mathfrak{P}_{[r]}$ constructed in this manner
satisfies $\tau\circ\tau=\hbox{id}$, which finishes the proof.

\section{Proof of Proposition \ref{prop_conditional_estimate}}
\label{sec:condest}
\begin{lemma}[Estimating local effects of conditioning]
\label{lemma_estimate_I}
Fix $0 \leq \alpha < \beta$ and an integer $s > s_r + r$. Then, for any partition $\cQ$ of $[r]$,
$\underline{h} \in \Delta_{\cQ}(\alpha, \beta)$ and $f \in \cA$,
\[
|\psi_{f,\underline{h}}(I) - \psi^{\cQ}_{f,\underline{h}}(I)| \ll_{r,s} e^{-(\beta \delta_r - r \alpha \sigma_s)} \, N_s(f)^{|I|},
\quad
\textrm{for all $I \subset [r]$},
\]
where the implicit constant depends only on $r$ and $s$.
\end{lemma}

\subsection{Proof of Proposition \ref{prop_conditional_estimate} assuming Lemma \ref{lemma_estimate_I}}

Fix a partition $\cP$ of $[r]$. Given $I \in \cP$, we set
\[
A_I = \psi_{f,\underline{h}}(I) - \psi_{f,\underline{h}}^{\cQ}(I) \qen B_I = \psi_{f,\underline{h}}^{\cQ}(I),
\]
so that we can write
\[
\widetilde{\psi}_{f,\underline{h}}(\cP) = \prod_{I \in \cP} (A_I + B_I)
\qand
\widetilde{\psi}_{f,\underline{h}}^{\cQ}(\cP) = \prod_{I \in \cP}  B_I.
\]
We claim that  
\begin{equation}
\label{prodineq}
\left| \prod_{I \in \cP} (A_I + B_I) - \prod_{I \in \cP} B_I \right| \leq 2^r C,
\end{equation}
where
\[
C = \max\left\{ \prod_{I \in S} \, \prod_{J \in T} A_I \, B_J \, : \, \cP = S \sqcup T, \enskip S \neq \emptyset \right\}.
\]
Indeed, if one expands the first product in \eqref{prodineq}, one ends up with $2^r$ terms, one of which equals the product of all of the
$B_I$'s. All other terms contains at least one $A_I$ for some $I \in \cP$ in them, and thus their absolute values are  trivially estimated from above by $C$. \\

By Lemma \ref{lemma_estimate_I}, we have for all $s > s_r + r$, 
\[
|A_I| \ll_{r,s} e^{-(\beta \delta_r - r \alpha \sigma_s)} N_s(f)^{|I|},
\]
and by \eqref{mono} and \eqref{unibnd},
\[
|B_I| \ll_s \prod_{J \in \cQ} \|f\|^{|I \cap J|}_{L^\infty} = \|f\|_{L^\infty}^{|I|} \ll_s N_s(f)^{|I|},
\]
for every $I \in \cP$. Hence,
\[
A_I B_J \ll_{r,s} e^{-(\beta \delta_r - r \alpha \sigma_s)} N_s(f)^{|I| + |J|}, \quad \textrm{for all $I, J \subset [r]$ with $I \neq \emptyset$}.
\]
Note that the bound in Proposition \ref{prop_conditional_estimate} is trivial (and useless) if $\beta \delta_r - r \alpha \sigma_s < 0$, so let us henceforth assume that $\beta \delta_r - r \alpha \sigma_s \geq 0$. We then get that
\[
|C| \ll_{r,s} e^{-(\beta \delta_r - r \alpha \sigma_s)} N_s(f)^{r},
\]
which finishes the proof.

\subsection{Proof of Lemma \ref{lemma_estimate_I}}

We assume that $0 \leq \alpha < \beta$ have been fixed once and for all, as well as a partition $\cQ$ of $[r]$, along with 
a subset $I \subset [r]$. Pick $f \in \cA$ and a tuple $\underline{h} = (h_1,\ldots,h_r) \in \Delta_{\cQ}(\alpha,\beta)$.
We recall that the latter means that
\begin{equation}
\label{abbounds}
d^{\cQ}(\underline{h}) \leq \alpha \qen d_{\cQ}(\underline{h}) > \beta,
\end{equation}
where $d^{\cQ}$ and $d_{\cQ}$ are defined in Subsection \ref{subsec:wellsep}. Let
\[
W_I = \big\{ J \in \cQ \, : \, I \cap J \neq \emptyset \big\},
\]
and choose, for every $J \in W_I$, an index $i_J \in I \cap J$. For every $J \in W_I$, we now set
\[
f_J = \prod_{j \in I \cap J} h_{i_J}^{-1}h_j \cdot f,
\]
and note that 
\[
\psi_{f,\underline{h}}(I \cap J) = \mu(h_{i_J} \cdot f_J) = \mu(f_J)
\qand
\psi_{f,\underline{h}}(I) = \mu(\prod_{J \in W_I} h_{i_J} \cdot f_J).
\]
In particular, by our convention that $\psi_{f,\underline{h}}(\emptyset) = 1$, 
\[
\psi_{f,\underline{h}}^{\cQ}(I) = \prod_{J \in \cQ} \psi_{f,\underline{h}}(I \cap J) = \prod_{J \in W_I} \mu(f_J).
\]
Since the action $H \acts (X,\mu)$ is assumed to be exponentially mixing of all orders, we conclude by \eqref{defmixk}
with $k = |W_I|$ and $\underline{h}^{W_I} = (h_J)_{J \in W_I}$, that for all $s' > s_k$, 
\begin{eqnarray}
\big| \psi_{f,\underline{h}}(I) - \psi^{\cQ}_{f,\underline{h}}(I) \big| 
&=&
\big|
\mu(\prod_{J \in W_I} h_{i_J} \cdot f_J) - \prod_{J \in W_I} \mu(f_J)
\big| \nonumber \\
&\ll_{k,s'} &
e^{-\delta_k d_k(\underline{h}^{W_I})} \, \prod_{J \in W_I} N_{s'}(f_J), \label{diff}
\end{eqnarray}
where 
\[
d_k(\underline{h}^{W_I}) = \min\big\{ d(h_J,h_{J'}) \, : \, J, J' \in W_I, \enskip J \neq J' \big\},
\]
and the implicit constants depend only on $k$ and $s'$. \\

Let us now estimate the norms $N_{s'}(f_J)$ for $J \in W_I$. We fix $J \in W_I$, and suppose that $I \cap J$ contains at least two elements so that we can write $I \cap J = J' \sqcup \{j\}$ for some $j \in I \cap J$ and $J' \subset I \cap J$. Then, by 
\eqref{Leibniz} and \eqref{grpbnd}, we have
\begin{align*}
N_{s'}(f_J) &= N_{s'}(f_{J'}  (h_{i_J}^{-1} h_j \cdot f)) \ll_{s'} 
N_{s'+1}(f_{J'})  N_{s'+1}(h_{i_J}^{-1} h_j \cdot f)\\
&\ll_{s'} 
\exp(\sigma_{s'+1} d(h_{i_J},h_j)) \, N_{s'+1}(f_{J'}) N_{s'+1}(f).
\end{align*}
If we iterate this argument as many times as there are elements in $I \cap J$, we arrive at the bound
\[
N_{s'}(f_J) \ll_{s'} \exp\left(|I \cap J| \sigma_{s'+|J|} d^{J}(\underline{h})\right) N_{s'+|J|}(f)^{|I \cap J|},
\]
where we used that $(\sigma_s)$ is increasing, and
thus, by \eqref{mono},
\begin{equation}
\label{prodNS}
\prod_{J \in W_I} N_{s'}(f_J) \ll_{s'} \exp\left( r \sigma_{s'+r} d^{\cQ}(\underline{h})\right)  N_{s'+r}(f)^{|I|},
\end{equation}
where $d^J$ and $d^{\cQ}$ are as in \eqref{defdupI} and \eqref{eq:dQ} respectively. Going back to 
\eqref{diff}, and using our assumption from Subsection \ref{standingass} that the sequences $(\sigma_s)$ and $(s_k)$ are increasing and that the 
sequence $(\delta_k)$ is decreasing, we conclude from \eqref{prodNS} that for all $s > s_r + r$, we have
\[
\big| \psi_{f,\underline{h}}(I) - \psi^{\cQ}_{f,\underline{h}}(I) \big| \ll_{r,s} e^{-(\delta_r d_k(\underline{h}^{W_I}) - r \sigma_{s} d^{\cQ}(\underline{h}))} \, N_{s}(f)^{|I|}.
\]
We have assumed that $\underline{h} \in \Delta_\cQ(\alpha,\beta)$, and thus 
\[
d^{\cQ}(\underline{h}) \leq \alpha \qen d_k(\underline{h}^{W_I}) \geq d_{\cQ}(\underline{h}) > \beta,
\]
whence 
\[
\delta_r d_k(\underline{h}^{W_I}) - r \sigma_{s} d^{\cQ}(\underline{h}) > \delta_r \beta - r \sigma_s \alpha,
\]
which finishes the proof.

\section{Proof of Proposition \ref{prop_exhaust}}

We retain the conventions and notations which were set up in Subsection \ref{subsec:wellsep}. In particular,
we fix an integer $r \geq 3$ throughout the section, and write $[r]$ for the set $\{1,\ldots,r\}$. 

\subsection{Passing to coarser partitions}

If $\cQ$ and 
$\cR$ are partitions of $[r]$, we say that $\cR$ is \textbf{coarser} than $\cQ$ if every partition element
in $\cR$ is a union of partition elements in $\cQ$, and \textbf{strictly coarser} if $\cR$ also has fewer
partition elements than $\cQ$. In other words, $\cR$ is strictly coarser than $\cQ$ if at least one partition
element in $\cR$ is the union of at least two partition elements from $\cQ$. In particular, the partition $\{[r]\}$ into one single block is strictly coarser
than any other partition of $[r]$, and every partition of $[r]$ with strictly less than $r$ partition elements is strictly
coarser than the partition $\cK_r$ of $[r]$ into points. \\

The following lemma summarizes the main inductive step in the proof of Proposition \ref{prop_exhaust}.

\begin{lemma}[Passing to coarser partitions]
\label{lemma_coarser}
Let $\cQ$ be a partition of $[r]$ with $|\cQ| \geq 2$. Fix $0 \leq \alpha < \beta$, and suppose that $\underline{h} \in H^r$ satisfies 
\[
d^{\cQ}(\underline{h}) \leq \alpha \qand d_{\cQ}(\underline{h}) \leq \beta.
\]
Then there exists a partition $\cR$ of $[r]$, strictly coarser than $\cQ$, such that 
$$
d^{\cR}(\underline{h}) < 3\beta.
$$
\end{lemma}

\subsection{Proof of Proposition \ref{prop_exhaust} assuming Proposition \ref{lemma_coarser}}

Let us fix a sequence 
\[
0 = \beta_o < \beta_1 < 3 \beta_1 < \beta_2 < 3 \beta_2 < \beta_3 < \ldots < \beta_{r-1} < 3 \beta_{r-1} < \beta_{r}.
\]
Pick an element $\underline{h} = (h_1,\ldots,h_r) \in H^r$. We wish to prove that either
\begin{itemize}
\item $\max_{i,j} d(h_i, h_j) \leq \beta_{r}$, or
\item there exist $0 \leq k < r-1$ and a partition $\cQ$ of $[r]$ with $|Q|\ge 2$ such that $\underline{h} \in \Delta_{\cQ}(3 \beta_k,\beta_{k+1})$.
\end{itemize}
This will be done in several steps. We first check whether $\underline{h} \in \Delta_{\cK_r}(0,\beta_1)$. If not, then
\[
d^{\cK_{r}}(\underline{h}) \leq 0 \qand d_{\cK_{r}}(\underline{h}) \leq \beta_1,
\]
and thus Lemma \ref{lemma_coarser} (applied to $\alpha = 0$ and $\beta = \beta_1$) implies that there exists a strictly coarser partition $\cQ_1$ than $\cK_{r}$ such that $d^{\cQ_1}(\underline{h}) \leq 3\beta_1$. If $|\cQ_1| = 1$ (that is, $\cQ_1=\{[r]\}$), then 
\[
\max_{i,j} d(h_i,h_j) \leq 3\beta_1 < \beta_r,
\]
and we are done, so let us assume that $|\cQ_1| \geq 2$. We now check whether $\underline{h} \in \Delta_{\cQ_1}(3\beta_1,\beta_2)$. If not, then 
\[
d^{\cQ_1}(\underline{h}) \leq 3\beta_1 \qand d_{\cQ_1}(\underline{h}) \leq \beta_2,
\]
and since $|\cQ_1| \geq 2$, Lemma \ref{lemma_coarser} (applied to $\alpha = 3\beta_1$ and $\beta = \beta_2$) 
implies that there exists a strictly coarser partition $\cQ_2$ than $\cQ_1$ such that 
$d^{\cQ_2}(\underline{h}) \leq 3\beta_2$. If $|\cQ_2| = 1$, then we again can conclude the argument as before, so we may assume that
$|\cQ_2| \geq 2$. 

If we continue like this, then we will have produced a chain $\cK_r,\cQ_1,\ldots,\cQ_m$ of strictly coarser partitions of 
$[r]$, which eventually must terminate at the trivial partition $\{[r]\}$ in no more than $r$ steps. At the $k$-th step, we check whether $\underline{h}$
belongs to $\Delta_{\cQ_k}(3\beta_k,\beta_{k+1})$. If this check fails for every $k$, then we conclude that
$\max_{i,j} d(h_i,h_j) \leq \beta_r$.

\subsection{Proof of Lemma \ref{lemma_coarser}}

Let $\cQ$ be a partition of $[r]$ with $|\cQ| \geq 2$ and $\underline{h} \in H^r$
satisfy $d^{\cQ}(\underline{h}) \leq \alpha$ and $d_{\cQ}(\underline{h}) \leq \beta$.
Since $|\cQ| \geq 2$, 
it follows from the second inequality that there exist atoms $I\ne J$ in $\cQ$
such that $d_{I,J}(\underline{h})\le \beta$.
We consider the partition
$\cR$ consisting of $I\cup J$ and $K\in \cQ\backslash \{I,J\}$ which is strictly coarser than $\cQ$.
Since $d_{I,J}(\underline{h})\le \beta$,
there exist $i_0\in I$ and $j_0\in J$ such that $d(h_{i_0},h_{j_0})\le \beta$.
Moreover, since $d^\cQ(\underline{h})\le \alpha$, we have 
$d^I(\underline{h})\le \alpha$ and 
$d(h_{i},h_{i_0})\le \alpha$ for all $i\in I$.
Similarly, we conclude that 
$d(h_{j},h_{j_0})\le \alpha$ for all $j\in J$. Hence, it follows that
for all $i\in I$ and $j\in J$,
$$
d(h_i,h_j)\le d(h_i,h_{i_0})+d(h_{i_0},h_{j_0})+d(h_{j_0},h_{j})\le \beta+2\alpha<3\beta.
$$
This proves that $d^{I\cup J}(\underline{h})< 3\beta$.
Additionally, for $K\in \cQ\backslash \{I,J\}$,
$$
d^{K}(\underline{h})\le d^{\cQ}(\underline{h})\le\alpha< 3\beta.
$$
Hence, we conclude that $d^{\cR}(\underline{h})< 3\beta$, as required.

\end{document}